\documentclass[11pt]{amsart}

\usepackage{latexsym}
\usepackage{amssymb}
\usepackage{amsmath}
\newfont{\formb}{cmb10.tfm}
\newfont{\formbs}{cmr8.tfm}

\newfont{\tit}{cmssbx10}
\newfont{\titt}{cmssbx10 scaled\magstep2}
\newfont{\lle}{cmtt12}
\numberwithin{equation}{section}

\def\NN{\mathbb N}

\def\RR{\mathbb R}

\def\R{\mathbb R}

 \def\1{\raisebox{2pt}{\rm{$\chi$}}}
 \newcommand{\eps}{{\varepsilon}}
 
 \def\1{\raisebox{2pt}{\rm{$\chi$}}}
 
\usepackage{enumerate}

\def\vint_#1{\mathchoice%
          {\mathop{\kern 0.2em\vrule width 0.6em height 0.69678ex depth -0.58065ex
                  \kern -0.8em \intop}\nolimits_{\kern -0.4em#1}}%
          {\mathop{\kern 0.1em\vrule width 0.5em height 0.69678ex depth -0.60387ex
                  \kern -0.6em \intop}\nolimits_{#1}}%
          {\mathop{\kern 0.1em\vrule width 0.5em height 0.69678ex
              depth -0.60387ex
                  \kern -0.6em \intop}\nolimits_{#1}}%
          {\mathop{\kern 0.1em\vrule width 0.5em height 0.69678ex depth -0.60387ex
                  \kern -0.6em \intop}\nolimits_{#1}}}
\def\vintslides_#1{\mathchoice%
          {\mathop{\kern 0.1em\vrule width 0.5em height 0.697ex depth -0.581ex
                  \kern -0.6em \intop}\nolimits_{\kern -0.4em#1}}%
          {\mathop{\kern 0.1em\vrule width 0.3em height 0.697ex depth -0.604ex
                  \kern -0.4em \intop}\nolimits_{#1}}%
          {\mathop{\kern 0.1em\vrule width 0.3em height 0.697ex depth -0.604ex
                  \kern -0.4em \intop}\nolimits_{#1}}%
          {\mathop{\kern 0.1em\vrule width 0.3em height 0.697ex depth -0.604ex
                  \kern -0.4em \intop}\nolimits_{#1}}}

\newcommand{\aveint}[2]{\mathchoice%
          {\mathop{\kern 0.2em\vrule width 0.6em height 0.69678ex depth -0.58065ex
                  \kern -0.8em \intop}\nolimits_{\kern -0.45em#1}^{#2}}%
          {\mathop{\kern 0.1em\vrule width 0.5em height 0.69678ex depth -0.60387ex
                  \kern -0.6em \intop}\nolimits_{#1}^{#2}}%
          {\mathop{\kern 0.1em\vrule width 0.5em height 0.69678ex depth -0.60387ex
                  \kern -0.6em \intop}\nolimits_{#1}^{#2}}%
          {\mathop{\kern 0.1em\vrule width 0.5em height 0.69678ex depth -0.60387ex
                  \kern -0.6em \intop}\nolimits_{#1}^{#2}}}

\begin{document}

\parskip=5pt

\def\a{{\bf a}}
\def\b{{\bf b}}
\def\z{{\bf z}}
\def\e{{\bf e}}
\def\g{{\bf g}}
\def\L{{\mathcal L}}
\def\zbar{{\bf \overline{z}}}
\def\ubar{\overline{u}}
\def\eps{\epsilon}
\def\step{\Delta t}
\def\fbar{\overline{f}}
\def\1{\raisebox{2pt}{\rm{$\chi$}}}

% THEOREM Environments --------------
\newtheorem{theorem}{Theorem}
\newtheorem{corollary}[theorem]{Corollary}
\newtheorem{lemma}[theorem]{Lemma}
\newtheorem{proposition}[theorem]{Proposici{\'o}n}
\newtheorem{definition}[theorem]{Definition}
\newtheorem{remark}[theorem]{Remark}
\newtheorem{example}[theorem]{Example}

\newtheorem{teo}{Teorema}[section]

%{\fontfamily{ptm}\selectfont this is typeset in times}

\bibliographystyle{plain}

\title[An obstacle problem for Tug-of-War games]
{\bf An obstacle problem for Tug-of-War games}

\author[J. J. Manfredi, J. D. Rossi
and S. J. Somersille]{Juan J. Manfredi, Julio D. Rossi and Stephanie J.
Somersille}

\address{Juan J. Manfredi
\hfill\break\indent
Department of Mathematics,
University of Pittsburgh. Pittsburgh, PA 15260. USA. }
\email{{\tt manfredi@pitt.edu}}

\address{Julio D. Rossi
\hfill\break\indent
Departamento de An{\'a}lisis Matem{\'a}tico, Universidad de Alicante, Ap 99, 03080, Alicante, SPAIN.
\hfill\break\indent {\rm and}
\hfill\break\indent
Departamento de Matem\'{a}tica, FCEyN Universidad de Buenos Aires,
Ciudad Universitaria, Pab 1 (1428),
Buenos Aires,
ARGENTINA. }
\email{{\tt julio.rossi@ua.es}}

\address{Stephanie J.
Somersille
\hfill\break\indent
Department of Mathematics,
Dartmouth College.
Hanover, NH 03755. USA. }
\email{{\tt Stephanie.J.Somersille@Dartmouth.edu}}


\begin{abstract}
We consider the  obstacle problem  for the infinity Laplace equation.
Given a Lipschitz boundary function and a Lipschitz obstacle we prove the
existence and uniqueness of a super infinity-harmonic function constrained to lie above
the obstacle which is infinity harmonic where it lies strictly above the obstacle.
Moreover, we show that this function is the limit of
value functions of a game we call obstacle tug-of-war.
\end{abstract}

\maketitle

\begin{section}{Introduction} \label{sec-intro}
\setcounter{equation}{0}

Recently, Peres, Schram, Sheffield and Wilson, \cite{PSSW},
discovered the relationship between limits of value functions of
Tug-of-War games and solutions to the infinity Laplacian. Also,
Peres and Sheffield, \cite{PS}, found a game whose values
approximate solutions to the $p-$Laplacian, see also the work by
Manfredi, Parvianen and Rossi,
\cite{MPR},
\cite{MPR2},
\cite{MPR3}, \cite{MPR4}, Bjorland, Caffarelli and Figalli,
\cite{BCF},
by Armstrong, Smart and Somersille, \cite{AM}, by Peres, Pet\'e and Somersille, \cite{Peres},
and by Antunov\'ic, Peres, Sheffield and Somersille, \cite{APSS}.

Our main goal in this work is to study the obstacle problem in
this context. That is, we propose a game that involves a function
$\Psi$, (the
obstacle), and is such that the value function of the game is above it.
We prove existence, uniqueness and some properties of the
value functions of this game and we find that a certain limit of these functions
is a viscosity solution of the obstacle problem  for the infinity
Laplacian.

Next, let us describe briefly the game in which we are interested.
The Tug-of-War game in \cite{PSSW} and described in detail below is a
two player zero sum game. In our case it is played
in a bounded domain $\Omega \subset \RR^N$ with a given ``boundary''
function $F:\Gamma
\mapsto
\RR$ (here $\Gamma$ is a neighbourhood of $\partial \Omega$ in
$\RR^N
\setminus
\Omega$). In our modification we also have an obstacle $\Psi :
\RR^N
\mapsto \RR$ such that $\Psi
\leq F$ in $\Gamma$. As in ordinary Tug-of-War if the boundary is reached at $x_n\in \Gamma$
 then Player I receives $F(x_n)$. However, in our case, Player I can opt to stop the game at any
position $x_n \in \Omega$ and
receive the payoff $\Psi (x_n)$.

This is much like the case in American options where investors can exercise the option at any time up to expiry
and accept a payoff equal to the intrinsic value which in our case is the obstacle. Or they may wait (continue to play) if the expected benefit of waiting is greater than the intrinsic value. Tug of war and the infinity Laplacian has applications to mass transport problems, control theory and economic modeling among others. Specifically, it is our belief that our results may have applications to the pricing of American options.

Our game is similar to investing in American Options in that optimal strategies are to stop where the value function agrees with the intrinsic value. Underlying our proof is the idea that Player I will choose to stop where his value function is equal to the obstacle. However, that he does so is not a requirement for any of our proofs.

%%%%%%%%%%%%%%%%%%%%%%%%%%%%%%%%%%%%%%
% Main Theorems
%%%%%%%%%%%%%%%%%%%%%%%%%%%%%%%%%%%%%%

We have the following results concerning properties of $u^\epsilon$ the value function of this
game where $\epsilon$ indicates the maximum size of each move:

\begin{theorem} \label{existencia.unicidad.teo.intro}
There exists a unique value of the game. This value is the
solution to the discrete obstacle problem; that is, it satisfies
$$
u^\epsilon(x)\geq \frac{1}{2}  \sup_{y\in {B}_\epsilon (x) } u^\epsilon(y) +
\frac{1}{2} \inf_{y \in  {B}_\epsilon (x) } u^\epsilon(y),
$$
lies above the obstacle in $\Omega$, coincides with $F$ in $\Gamma$,  and
is such that the
above inequality is an equality where $u^\epsilon$ lies strictly above the obstacle.
\end{theorem}

In addition, a comparison principle holds.
\begin{lemma} \label{comp.princ}
Let $u_1^\epsilon$, $u_2^\epsilon$ be
values of the $\epsilon$ games with boundary functions $F_1$, $F_2$ and
obstacles $\Psi_1$, $\Psi_2$ respectively.
If $F_1 \geq F_2$ and $\Psi_1
\geq
\Psi_2$ , then $u_1^\epsilon \geq u_2^\epsilon$.
\end{lemma}

Moreover, the value function of the game satisfies the following
Lewy-Stampacchia inequalities
\begin{lemma} \label{lewy.stamp}
We have
$$
\displaystyle 0 \leq u^\epsilon (x) - \frac{1}{2} \left( \sup_{B_\epsilon (x) }
u^\epsilon  +
\inf_{B_\epsilon (x) } u^\epsilon \right)
\displaystyle \leq
\left[ \ \Psi (x) - \frac{1}{2} \left( \sup_{B_\epsilon (x) }
\Psi  +
\inf_{B_\epsilon (x) } \Psi \right) \right]_+.
$$
\end{lemma}

Here we use the notation $[A(x)]_+=\max\{A(x), 0 \}$.

Finally, if the obstacle $\Psi$ is   Lipschitz, then
the value function is Lipschitz with respect to the discrete distance
$d_\epsilon(x,y) = \epsilon [\frac{|x-y|}{\epsilon}+1]$ (by $[\cdot]$ we
denote the integer part).
\begin{lemma} \label{d.lip}
If the obstacle $\Psi$ Lipschitz, then
there exists a constant $C$, independent of $\epsilon$, such that
the value function $u^\epsilon$ satisfies
$$
|u^\epsilon (x) - u^\epsilon (y) | \leq C d_\epsilon (x, y).
$$
\end{lemma}

Concerning the limit as $\epsilon \to 0$ of these value functions
we have the following result:

\begin{theorem} \label{convergencia.teo.intro}
We have that, as $\epsilon \to 0$,
$$
u^\epsilon \to u
$$
uniformly. The limit $u$ is the unique viscosity solution to the obstacle
problem for the infinity Laplacian, that is, it is the unique super-infinity harmonic function;  i.e., functions that satisfy
\begin{equation*}
-\Delta_\infty u = - \langle D^2 u \frac{Du}{|Du|}, \frac{Du}{|Du|} \rangle \geq 0,
\end{equation*}
in the viscosity sense,  is above the obstacle $\Psi$ in $\Omega$, takes the boundary value, $F$ on $\Gamma$
and is infinity harmonic where it lies strictly above the obstacle.
\end{theorem}

There is a close connection between $p-$harmonic functions and infinity harmonic functions. Indeed, if we pass to the limit, as $p\to \infty$, in a sequence $(u_p)$ of $p-$harmonic functions (in the viscosity sense, see \cite{CIL}), that is, solutions of $\Delta_p u_p =0$,
with given boundary values, the limit exists (in the uniform topology) and is a solution of the infinity Laplace equation (see \cite{BBM})
$$-\Delta'_\infty u =-\langle D^2 u Du, Du \rangle =- \sum_{i,j=1}^{d} \frac{\partial u}{\partial x_i} \frac{\partial u}{\partial x_j} \frac{\partial^2 u}{\partial x_i \partial x_j}=0.$$
The infinity Laplacian is connected with the optimal Lipschitz extension problem \cite{ACJ}, and arises also in the context of mass transportation problems and several other applications, such as image reconstruction and enhancement
\cite{cms}. Note that here we have normalized the operator and consider $
-\Delta_\infty u = - \langle D^2 u \frac{Du}{|Du|}, \frac{Du}{|Du|} \rangle$. Both equations turns out to be equivalent, in the sense that they have the same viscosity solutions, see \cite{PSSW}.

On the other hand,
the obstacle problem for elliptic operators has been extensively studied. In the classical approach one seeks to minimize the energy
$E(u) = \int_\Omega |D u|^2$
among the functions that coincide with a given function $F$ at the boundary of $\Omega \subset \R^d$ and remain above a prescribed obstacle $\Psi$. Such a problem is motivated by the description of the equilibrium position of a membrane (the graph of the solution) attached at level $F$ along the boundary of $\Omega$ and that is forced to remain above the obstacle in the interior of $\Omega$.
Many of the results obtained for the Laplacian were generalized for the $p-$Laplacian whose energy functional is given by
$E(u) = \int_\Omega |D u|^p$.

However, the infinity Laplacian is not variational (hence no energy methods are directly available). One may rely on methods from potential theory (Perron method),
on limit procedures like the ones described here,  or one can take the limit as $p\to \infty$ in the obstacle problem for the $p-$Laplacian and obtain a solution to the obstacle problem for the infinity Laplacian, see \cite{RTU} for example.

\medskip

The paper is organized as follows: in Section \ref{sec-Descrip} we describe the Tug-of-War game;
in Section \ref{sec-Proper} we collect some properties of the value function of the game; finally, in Section~\ref{sec-limit} we deal with the limit as $\eps \to 0$ and find a proof of existence of a solution to the obstacle problem for the infinity Laplacian based on Tug-of-War games. We also discuss the convergence of the contact sets. (The sets where the value function equals the obstacle).
\end{section}

%%%%%%%%%%%%%%%%%%%%%%%%%%%%%%%%%%%%%%%%%%%%%%%%%%%%%
\begin{section}{Description of the game} \label{sec-Descrip}
\setcounter{equation}{0}

%%%%%%%%%%%%%%%%%%%%%%%%%%%%%%%%%%%%%%%%%%%%%%%%%%%%%

\subsection{Description of the game} \label{subsect-descrip.game}
The game that we describe
below is called a leavable game. Some leavable games are described in
\cite{MS} Chapter 7. Tug-of-War is developed in \cite{PSSW}.

Tug-of-War is a two-person, zero-sum game, in other words, two players
are in contest and the total earnings of one are the losses of the
other. Hence, one of them, whom we call Player I, plays trying to maximize
his expected outcome, while the other, Player II, is trying to
minimize Player I's outcome (or, since the game is zero-sum, to
maximize his own outcome). In this Tug-of-War leavable game
Player I can decide to end the game before the boundary is reached
i.e. his strategy includes stopping rule.

Now, let us describe the game more precisely. Let
$\Omega
\subset
\RR^N$ be a bounded smooth domain. For a fixed $\gamma >0$,
consider a strip around the boundary $\Gamma\subset\RR^N
\setminus\Omega$ given by
$$
\Gamma = \left\{ x \in \RR^N \setminus \Omega \ : \ \mbox{dist} (x ,\partial \Omega)
\leq \gamma \right\}.
$$
Let $F:\Gamma\rightarrow\R$ be a Lipschitz continuous function
(the final payoff). In addition we have a function $\Psi: \RR^N
\mapsto \RR$ (the obstacle) such that
$$
\Psi \leq F \qquad \mbox{in } \Gamma.
$$
Note that any neighborhood of $\partial \Omega$ in $\RR^N \setminus \Omega$ contains a strip of this form.

The rules of the game are as follows:
 At
an initial time a token is placed at a point $x_0\in\Omega $ and
we fix $\epsilon \in (0, \gamma]$. Then, a (fair) coin is tossed and the
winner of the toss is allowed to move the game position to any
$x_1\in
{B}_\epsilon(x_0)$. At each turn, the coin is tossed again, and the
winner of the toss chooses a new game state $x_k\in
{B}_\epsilon(x_{k-1})$. Once the token has reached some
$x_\tau\in\Gamma$, the game ends and Player I earns $F(x_\tau)$
(while Player II earns $-F(x_\tau)$). This is the reason why we
will refer to $F$ as the
\emph{final payoff function}. In addition, at every position $x_n$, Player I
is allowed choose to end the game earning $\Psi (x_n)$
(while Player II earns $-\Psi(x_n)$). We will call $\Psi$ the
\emph{obstacle function}. This procedure
yields a sequence of game states $x_0,x_1,x_2,\ldots, x_\tau$,
where every $x_k$ except $x_0$ are random variables, depending on
the coin tosses, the strategies (defined below) adopted by the players and the
stopping rule chosen by Player I.

Note that the relevant values of $F$ are those taken in the set
$$
\Gamma_\epsilon = \bigcup_{x\in {\Omega}}
{B}_\epsilon(x) \cap \Gamma
$$
since those are the the points at which the game could end.

Next, we give a precise definition of the {\it value of the game}.
To this end we have to introduce some notation and put the game
into its normal or strategic form (see \cite{PS}). The initial state
$x_0\in \Omega$ is known to both players (public knowledge). Each
player $i$ chooses an \emph{action} $a_{0}^i\in
{B}_\epsilon(x_0)$ which is announced to the other player;
this defines an action profile $a_0=\{a_{0}^1,a_{0}^2\}\in
{B}_\epsilon(x_0)
\times {B}_\epsilon(x_0)$. Then, the
new state $x_1\in {B}_\epsilon(x_0)$ is selected
according to a probability distribution $p(\cdot|x_0,a_0)$ in
$\Omega$ which, in our case, is given by the fair coin toss. In
addition, Player I, chooses a stop rule $\tau$ (here $\tau$ takes
values in $\NN$, and determines that the game ends at step $\tau$). At stage $k$, knowing the history
$h_k=(x_0,a_0,x_1,a_1,\ldots,a_{k-1},x_k)$, (the sequence of
states and actions up to that stage), Player I chooses to end the
game or to continue according to the stoping rule $\tau$ (that is
a mapping from histories $h_k$ to $\NN$, if $\tau \neq k$ the game
continues, while, if $\tau = k$ the game ends), if she decides to
continue, each player $i$ chooses an action $a_{k}^i$. If the game
ends at time $j $ (the game ends if the position $x_j$ belongs to
$\Gamma$ or if the stopping rule for the first player applies), we set
$x_m = x_j$ and $a_m=x_j$ for $j\leq m$.

Denote $H_k=(\Omega\cup\Gamma)^k=\big(
(\Omega\cup\Gamma)\times
(\Omega\cup\Gamma)\times \ldots \times(\Omega\cup\Gamma)\big)$, the set of
\emph{histories up to stage $k$}, and by $H_\infty
=\bigcup_{k\geq1}H_k$ the set of all histories. Notice that $H_k$,
as a product space, has a measurable structure. The \emph{complete
history space} $H_\infty$ is the set of plays defined as infinite
sequences $(x_0,a_0,\ldots,a_{k-1},x_k,\ldots)$ endowed with the
product topology. Then, the final payoff for Player I, defined by
$$\tilde{F} (x) = \left\{
\begin{array}{ll}
F(x), \qquad x \in \Gamma, \\
\Psi (x), \qquad x \in \Omega,
\end{array}\right.
$$
induces a Borel-measurable function on $H_\infty$. A
\emph{strategy} ${S}_i=\{S_{i}^k\}_k$ for Player $i$,  is a
sequence of mappings from histories to actions, such that
${S}_{i}^{k}$ is a Borel-measurable mapping that maps histories
ending with $x_k$ to elements of $ {B}_\epsilon(x_k)$ (roughly
speaking, at every stage the strategy gives the next movement for
the player, provided he win the coin toss, as a function of the
current state and the past history). A
\emph{stopping rule} for Player I is a stopping time, $\tau$, from histories $H$ to
$\NN$ that is a finite everywhere, Borel-measurable, such that for
every $k$ the set $\{\tau=k\}$ belongs to the sigma field
generated by the coordinate functions $X_1,...,X_k$ of $H$ (if
$\tau\neq k$ the game continues, while if $\tau=k$ the game ends).

The initial state $x_0$, a stopping rule $\tau$ and a profile of
strategies $\{S_I,S_{II}\}$ define (by Kolmogorov's extension
theorem) a unique probability $\mathbb{P}_{\tau,S_I,S_{II}}^{x_0}$
on the space of plays $H_\infty$. We denote by
$\mathbb{E}_{\tau,S_I,S_{II}}^{x_0}$ the corresponding
expectation.

Then, if $\tau$ denotes the stopping rule for Player I and $S_I$
and $S_{II}$ denote the strategies adopted by Player I and
Player~II respectively, we define the expected payoff for Player I
as
\[
V_{x_0,I}(\tau,S_I,S_{II})=
\left\{
\begin{array}{ll}
\displaystyle \mathbb{E}_{\tau,S_I,S_{II}}^{x_0}[\tilde{F}(x_\tau)],\quad &
\text{if the game terminates a.s.}\\
\displaystyle -\infty,\quad & \text{otherwise.}
\end{array}
\right.
\]
Analogously, we define the expected payout for Player II as
\[
V_{x_0,II}(t,S_I,S_{II})=
\left\{
\begin{array}{ll}
\displaystyle \mathbb{E}_{\tau,
S_I,S_{II}}^{x_0}[\tilde{F}(x_\tau)],\quad & \text{if the game terminates a.s.}\\
\displaystyle +\infty,\quad & \text{otherwise.}
\end{array}
\right.
\]
Finally, we can define the $\epsilon$-value of the game for Player
I as
\[
u_{I}^\epsilon(x_0)=\sup_{\tau,S_I}\inf_{S_{II}}\,V_{x_0,I}(\tau,S_I,S_{II}),
\]
while the $\epsilon$-value of the game for Player II is defined as
\[
u_{II}^\epsilon(x_0)=\inf_{S_{II}}\sup_{\tau,S_I}\,V_{x_0,II}(\tau,S_I,S_{II}).
\]
In some sense, $u_I^\epsilon(x_0)$, $u_{II}^\epsilon(x_0)$ are the
least possible outcomes that each player expects when the
$\epsilon$-game starts at $x_0$. Notice that, as in \cite{PSSW},
we penalize severely the games that never end.

In \cite{MS} it is shown that, under very general hypotheses that
are fulfilled in the present setting, $u_{I}^\epsilon=
u_{II}^\epsilon :=u^\epsilon$. The function $u^\epsilon$ is called
the value of the $\epsilon$-Tug-of-War game.

%%%%%%%%%%%%%%%%%%%%%%%%%%%%%%%%%%%%%%%%%%%%%%%%%%%%%%%%

\subsection{Dynamic Programming Principle} \label{subsect-Properties.game}

%%%%%%%%%%%%%%%%%%%%%%%%%%%%%%%%%%%%%%%%%%%%%%%%%%%%%%%%

For $x\in \Omega$, looking at the stopping strategy and the outcome of the first coin toss, we immediately
get the following lemma, that says that the values of the game satisfy a Dynamic Programming Principle (DPP) formula (see also \cite{MS} for similar Dynamic Programming Principles).

\begin{lemma}
{\rm (DPP)}
The value functions $u^\epsilon_I$ and $u^\epsilon_{II}$ satisfy
$$u^\epsilon(x)=\max \left\{ \Psi(x), \frac{1}{2}  \sup_{y\in {B}_\epsilon (x) } u^\epsilon (y)+
\frac{1}{2} \inf_{y\in {B}_\epsilon (x) } u^\epsilon (y) \right\} \qquad \forall x \in \Omega,$$
and
$$
u^\epsilon(x)=F(x),\qquad \forall x \in \Gamma.
$$
\end{lemma}

This immediately implies that $u^\epsilon_I$ and $u^\epsilon_{II}$ satisfy
the following formulation of the discrete obstacle problem:
$$
\left\{
\begin{array}{ll}
u^\epsilon (x) = F (x),  & \mbox{in } \Gamma, \\[6pt]
u^\epsilon (x) \geq \Psi (x),  & \mbox{ in } \Omega, \\[6pt]
\displaystyle  u^\epsilon (x) \geq \frac{1}{2} \left( \sup_{y\in {B}_\epsilon (x) } u^\epsilon (y) +
\inf_{y\in {B}_\epsilon (x) } u^\epsilon (y)\right) ,  & \mbox{ in }
\Omega , \\[6pt]
\displaystyle  u^\epsilon (x) = \frac{1}{2} \left( \sup_{y\in {B}_\epsilon (x) } u^\epsilon(y) +
\inf_{y\in {B}_\epsilon (x) } u^\epsilon(y)\right) ,  & \mbox{ in }
\Omega \setminus A^{u^\epsilon} .
\end{array}
\right.
$$
Here $A^{u^\epsilon}$ is the coincidence set, that is, the set where $u^\epsilon =\Psi$ in $\Omega$.

We call the last equality discrete $\epsilon$ infinity harmonic. We call the last inequality discrete $\epsilon$ infinity super harmonic. Discrete $\epsilon$ infinity subharmonic is defined analogously.
\end{section}

%%%%%%%%%%%%%%%%%%%%%%%%%%%%%%%%%%%%%%%%%%%%%%%%%%%%

\begin{section}{Properties of the game value functions} \label{sec-Proper}
\setcounter{equation}{0}

%%%%%%%%%%%%%%%%%%%%%%%%%%%%%%%%%%%%%%%%%%%%%%%%%%%%

Here we prove the  existence of a value function for the game, Theorem \ref{existencia.unicidad.teo.intro},  and our comparison principle Lemma
\ref{comp.princ}.

To prove these we need some lemmas. We will show that $u_I^\epsilon$ is the smallest supersolutions that satisfies our conditions and $u_{II}^\epsilon$ is, in some sense, the largest subsolution.

\begin{remark} \label{nota7}{\rm Note that $u_I^\epsilon$ and $u_{II}^\epsilon$ are at least as large as the corresponding ordinary tug of war game with ``boundary" $Y= \Gamma \cup A^{u}$ where $A^u$ is the corresponding contact set i.e. $A^u=A^{u_I}$ or $A^u=A^{u_{II}}$. More precisely, let $A^u$ be the contact set of $u$
and let $Y=\Gamma \cup A^u$. Let $\hat{F}:Y\to\RR$ be the Lipschitz function
$$\hat{F} (x) = \left\{
\begin{array}{ll}
F(x), \qquad x \in  \Gamma, \\
\Psi (x), \qquad x \in  A^u .
\end{array}\right.
$$
Notice $\hat{F}$ is well defined because if $\Gamma \cap A^u$ is nonempty then $F=\Psi$ there. Then we have the inequality $ w^\epsilon \leq u$ where $w^\epsilon$ is the value function for the ordinary tug-of-war game in this setting.

We have $ w^\epsilon\leq u$ because Player I could always play as if he were in this ordinary tug-of-war situation so he can do at least as well in the obstacle game.
}
\end{remark}

\begin{lemma} \label{lema.supersol}
Let $v$ be a supersolution to the DPP, that is, a function that satisfies
$$
\left\{\begin{array}{ll}
v(x) \geq F (x),  & \mbox{in } \Gamma, \\[8pt]
v(x) \geq \Psi (x),  & \mbox{in } \Omega, \\[8pt]
\displaystyle v(x) \geq \frac{1}{2} \left( \sup_{{B}_\epsilon (x) } v(y) +
\inf_{{B}_\epsilon (x) } v(y)\right) ,  & \mbox{ in }
\Omega ,
\end{array}
\right.
$$
then we have
$$
u_I^\epsilon (x)\leq v(x).
$$
\end{lemma}

\begin{proof}
If $x_0 \in A^{u_I^\epsilon}$ then $u_I^\epsilon(x_0)=\Psi (x_0) \leq v(x_0)$. So we assume $x_0 \notin A^{u_I^\epsilon}$. In $\Omega \setminus A^{u_I^\epsilon}$ we have
$$u_I^\epsilon(x)=\frac{1}{2} \sup_{{B}_\epsilon (x) } u_I^\epsilon(y) + \frac{1}{2}
\inf_{{B}_\epsilon (x) } u_I^\epsilon(y) > \Psi(x).$$
Let $w^\epsilon$ be the tug-of-war game, without obstacle, described in remark \ref{nota7} (with $A^u=A^{u_I^\epsilon}$).
%We may assume that Player I chooses a stopping time that does not end the game in $\Omega \setminus A^{u_I}$ since any such strategy would lead to a lower expected payoff than the same strategy with a stopping time that does not end the game there.
Thus, since $u_I^\epsilon$ is $\epsilon$ discrete infinity harmonic in $\Omega \setminus A^{u_I^\epsilon}$, and $w^\epsilon$ is the unique such function by \cite{PSSW}, we have $u_I^\epsilon=w^\epsilon \leq v$.
\end{proof}

\begin{lemma} \label{lema.subsol}
Let $v$ be a subsolution away from the obstacle which also lies above the obstacle. That is, a function that satisfies
$$
\left\{
\begin{array}{ll}
v(x) \leq F (x),  & \mbox{in } \partial \Omega, \\[8pt]
v(x) \geq \Psi(x),  & \mbox{in } \Omega, \\[8pt]
\displaystyle v(x) \leq \frac{1}{2} \left( \sup_{{B}_\epsilon (x) } v(y) +
\inf_{{B}_\epsilon (x) } v(y)\right) ,  & \mbox{ in }
\Omega \setminus A^v ,
\end{array}
\right.
$$
then we have
$$
v(x) \leq u_{II}^\epsilon (x).
$$
\end{lemma}

\begin{proof}
For  $x_0 \in A^v$ we have $v(x_0)=\Psi (x_0) \leq u_{II}^\epsilon (x_0)$.

Assume $x_0 \in \Omega \setminus A^v$.
Let $w^\epsilon$ be the value of discrete $\epsilon$ tug of war (without obstacle) on $\Omega$ as described in the remark \ref{nota7} with $A=A^v$ and $\hat{F}$ a Lipshitz extension of $F$ to $A^v$ such that $v \leq \hat{F} \leq u_{II}^\epsilon$.

Since $v\leq \hat{F}$ and $v$ is subharmonic in $\Omega\setminus{A^v}$, by \cite{PSSW} we have that $$v(x_0)\leq w^\epsilon(x_0).$$
We have $$w^\epsilon\leq u_{II}^\epsilon$$ on all of $\Omega$ by the remark. Therefore $$v\leq u_{II}^\epsilon$$ on $\Omega \setminus{A^v}$ as well.
\end{proof}

Now we are ready to prove existence of a unique value of the game.

\begin{proof}[Proof of Theorem~\ref{existencia.unicidad.teo.intro}]
We always have $u_I^\epsilon \leq u_{II}^\epsilon$.
%(this is also follows from Lemma~\ref{lema.supersol} or ~\ref{lema.subsol}).
For $x_0\in A^{u_{II}^\epsilon}$ we have $$ u_{II}^\epsilon(x_0)=\Psi(x_0) \leq u_{I}^\epsilon(x_0).$$

Assume $x_0\in  \Omega \setminus A^ {u_{II}^\epsilon}$.
Let $w^\epsilon$ be as in the remark with $A=A^{u_{II}^\epsilon}$.
We have that
$ u_{II}^\epsilon$ is $\epsilon$ game harmonic on $\Omega \setminus A^{u_{II}^\epsilon}$ therefore, by \cite{PSSW}, $u_{II}^\epsilon=w^\epsilon$. And, by the remark, we have $$w^\epsilon (x_0) \leq u_{I}^\epsilon (x_0).$$
\end{proof}

We will now drop the subscripts and let $u^\epsilon=u^\epsilon_I=u^\epsilon_{II}$.

We now prove a small lemma that will be needed for the proof of Lemma ~\ref{lewy.stamp}, the Lewy-Stampaccia Lemma.

\begin{lemma} \label{inequality}
We have that $\Psi$ is discrete $\epsilon$ infinity superharmonic on the coincidence set, i.e.
$$
\Psi (x) - \frac{1}{2} \left( \sup_{{B}_\epsilon (x) }
\Psi (y) +
\inf_{{B}_\epsilon (x) } \Psi (y)\right) \geq 0, \qquad x \in A^{u^\epsilon}.
$$
\end{lemma}

\begin{proof}
In the set $A^{u^\epsilon}$ we have
\[
\begin{split}
\Psi(x) = u^\epsilon(x) &= \max \left\{\Psi(x), \frac{1}{2} ( \sup_{{B}_\epsilon (x) } u^\epsilon (y)+
 \inf_{{B}_\epsilon (x) } u^\epsilon(y)  )\right\}
\\ & \geq  \frac{1}{2} ( \sup_{{B}_\epsilon (x) } u^\epsilon (y)+
 \inf_{{B}_\epsilon (x) } u^\epsilon(y)  )\\
&
\geq  \frac{1}{2} ( \sup_{{B}_\epsilon (x) } \Psi (y)+
 \inf_{{B}_\epsilon (x) } \Psi(y)  ).
\end{split}
\]
The last inequality holds because
$u^\epsilon \geq \Psi$ on $\Omega$.
\end{proof}

\begin{proof}[Proof of Lemma \ref{lewy.stamp}]
The first inequality is immediate from the dynamic programming principle. If $x\not\in A^{u^\epsilon}$ then the second inequality is clear since $u^\epsilon$ is discrete $\epsilon$ infinity harmonic there. (Also but the dynamic programming principle).
If $x\in A^{u^\epsilon}$, then, from the fact that $u^\epsilon
\geq
\Psi$ on $\Omega$
and $u^\epsilon (x) = \Psi (x)$ for $x \in
A^{u^\epsilon}$, we get the last inequality.
\end{proof}
\end{section}

%%%%%%%%%%%%%%%%%%%%%%%%%%%%%%%%%%%%%%%%%%%%%%%%%%%%%%

\begin{section}{Limit as $\epsilon \to 0$.} \label{sec-limit}
\setcounter{equation}{0}

%%%%%%%%%%%%%%%%%%%%%%%%%%%%%%%%%%%%%%%%%%%%%%%%%%%%%%

In this section we prove our main result Theorem ~\ref{convergencia.teo.intro} regarding the limit of the game value functions.

Recall the discrete distance is given by
$d_\epsilon(x,y)=
\epsilon \lceil{\frac{|x-y|}{\epsilon}}\rceil $.
The proof of Lemma \ref{d.lip}, that the game value function is Lipschitz with respect to the $d_\epsilon$ metric, is now immediate
from \cite{PSSW}. The proof of the Uniform Lipschitz Lemma 3.5 from \cite{PSSW} shows that the Lipschitz constant depends only on the Lipschitz constants of $F$ and $\Psi$.

We are now in a position  to apply the following variant of the Arzela-Ascoli
Lemma. For its proof we refer the reader to \cite{MPR3} Lemma 4.2.

\begin{lemma}\label{lem.ascoli.arzela} Fix $\delta>0$. Let $\{u^\epsilon : \overline{\Omega}
\to \RR,\ \delta \geq \epsilon>0\}$ be a set of functions such that
\begin{enumerate}
\item there exists $C>0$ so that $|u^\eps(x)|<C$ for
    every $\delta \geq \eps>0$ and every $x \in \overline{\Omega}$,
\item \label{cond:2}
given $\eta>0$ there are constants
    $r_0$ and $\epsilon_0$ such that for every $\epsilon < \epsilon_0$
    and any $x, y \in \overline{\Omega}$ with $|x - y | < r_0 $
    it holds
$$
|u^\epsilon (x) - u^\epsilon (y)| < \eta.
$$
\end{enumerate}
Then, there exists  a uniformly continuous function $u:
\overline{\Omega} \to \RR$ and a subsequence denoted by
$\{u^{\epsilon_j} \}$ such that
\[
\begin{split}
u^{\epsilon_j }\to u \qquad\textrm{ uniformly
in}\quad\overline{\Omega},
\end{split}
\]
as $j\to \infty$.
\end{lemma}

\begin{theorem} \label{lema.conver}
If $F$ and $\Psi$ are Lipschitz continuous functions then there exists a subsequence of the values of the game
$u^{\epsilon_j}$ that converges uniformly to a continuous function $u$ in $\overline{\Omega}$,
$$
\lim_{\epsilon_j \to 0} u^{\epsilon_j} = u.
$$
\end{theorem}

\begin{proof}  Lemma \ref{lem.ascoli.arzela} can be applied since condition 1 holds with
\newline $C=\max \{F(x),\psi(x)\}$ and for condition 2 we can take, for instance, $\eps_0=\delta$ and  $r_0=\frac{\eta}{L}$, where $L$ is the Lipschitz constant of $u^{\delta}$ with respect to $d_\delta$ which does not depend on $\delta$.
\end{proof}

\begin{remark} {\rm
If we assume that $\Psi$ is $C^2$ the Lewy-Stampacchia
estimate gives that there exists $K>0$ such that
$$
0\leq u^\epsilon (x) - \frac{1}{2} \left( \sup_{B_\epsilon (x) }
u^\epsilon (y) +
\inf_{B_\epsilon (x) } u^\epsilon (y)\right) \leq K \epsilon^2.
$$
From here it follows that we can get estimates on the values
$u^\epsilon$ that allow us to pass to the limit using the variant of the well-known
Arzela-Ascoli type result from \cite{MPR3}, Lemma~\ref{lem.ascoli.arzela}.
}
\end{remark}

Next we prove that this uniform limit of the values of the game is the viscosity solution of the obstacle problem for the infinity Laplacian.
We are now ready to prove our main result, Theorem ~\ref{convergencia.teo.intro}

We prove this theorem by comparing the $u^{\epsilon_j}$ to appropriately defined discrete harmonic $v^{\epsilon_j}$ with fixed boundary conditions which we know converge to an infinity harmonic function. We prove uniqueness by proving that our limit is the least super harmonic function that lies above the obstacle.

\begin{proof} [Proof of Theorem \ref{convergencia.teo.intro}]
We first prove that our limit is infinity harmonic where it lies strictly above the obstacle. Passing to a subsequence if necessary we let $u= \lim u^{\epsilon_j}$. Fix $x_0 \in \Omega \setminus A^u$, and choose $r$ such that  $B_r(x_0)$ is included in the set $ \Omega\setminus A^u$ (which is open). Given $\delta$, for $\epsilon_j$ small enough, we have that $|u-u^{\epsilon_j}|<\delta$ on $B_r(x_0)$.

Define $v^{\epsilon_j}$ to be the discrete $\epsilon_j$ harmonic function that agrees with $u$ on $\partial B_r(x_0)$. Then we have
$v^{\epsilon_j} - \delta < u^{\epsilon_j} < v^{\epsilon_j} + \delta$ on $\partial B_r(x_0)$. Since $v^{\epsilon_j} - \delta$ and $v^{\epsilon_j} + \delta$ are also discrete harmonic on $B_r(x_0)$ with lower and higher boundary values respectively than $u^{\epsilon_j} $, with the help of the comparison principle for discrete harmonic functions we get the inequalities on all of $B_r(x_0)$.
By \cite{MPR3} and \cite{PSSW}, we know that $v^{\epsilon_j}$ converges uniformly to an infinity harmonic function $v$ on $B_r(x_0)$. Therefore, by the sandwich lemma and sending $\delta \to 0$, we have that $u=v$ thus $u$ is infinity harmonic on $B_r(x_0)$ and therefore $u$ is infinity harmonic on $\Omega \setminus A^u$.

Now let $x_0\in A^u$. We have that $\Psi$ is $C^2$ and for $\epsilon$ small enough satisfies
$$
\Psi(x_0)\geq \frac{1}{2}  \sup_{y\in {B}_\epsilon (x_0) } \Psi(y) +
\frac{1}{2} \inf_{y \in  {B}_\epsilon (x_0) } \Psi(y),
$$
therefore, in $A^u$ we have that
\begin{equation*}
-\Delta_\infty \Psi  = - \langle D^2 u \frac{D\Psi}{|D\Psi|}, \frac{D\Psi}{|D\Psi|} \rangle \geq 0.
\end{equation*}

From this and the proof of Lemma \ref{lewy.stamp} we have that $u$ also satisfies this inequality on $A^u$.

Therefore we have that the uniform limit of a subsequence of the values of the game, $u$, satisfies
$$
-\Delta_\infty u = 0, \quad \mbox{ in } \Omega\setminus A^u, \qquad \mbox{ and } \qquad -\Delta_\infty u \geq 0, \quad \mbox{ in } \Omega.
$$

We now prove uniqueness.
We define $u_\infty$ to be the least  infinity superharmonic function that is
above the obstacle and the boundary function.

Since $u$ is infinity superharmonic and above the obstacle and boundary function
we get the inequalities
$
u \geq u_\infty \geq \Psi
$
from which we have
$$
A^u \subset A^{u_\infty}.
$$
so on $A^u$ we have $u=u_\infty$.

Now, in the set $\Omega \setminus A^u$, $u$ is a solution to
$-\Delta_\infty u =0$ and $u_\infty$ is a supersolution with the same boundary values
($u_\infty = u = F$ on $\Gamma$ and
$u_\infty = u = \Psi$ on $A^u$). Therefore, the comparison principle for $\Delta_\infty$
implies that
$$
u_\infty \geq u
$$
in  $\Omega \setminus A^u$.
And then we conclude that
$$
u_\infty = u.
$$

Since we have uniqueness of the limit, the whole sequence $u^\epsilon$ converges uniformly.
\end{proof}

%%%%%%%%%%%%%%%%%%%%%%%%%%%%%%%%%%%%%%%%%%%%%%%%%%

\subsection{Convergence of the contact sets} \label{subsect-ContactSet}

%%%%%%%%%%%%%%%%%%%%%%%%%%%%%%%%%%%%%%%%%%%%%%%%%%%

We now simplify notation slightly and let $A^{\epsilon_j}:=A^{u^{\epsilon_j}}$ and we discuss the
convergence of the contact sets of the $u^{\epsilon_j}$ to the contact set of the limit function $u$.

We define $$\limsup_{j\to \infty} A^{\eps_j}=\bigcap_{p=1}^\infty \bigcup_{j=p}^\infty A^{\epsilon_j}\qquad
\mbox{and}  \qquad\liminf_{j\to \infty} A^{\eps_j}=\bigcup_{p=1}^\infty \bigcap_{j=p}^\infty A^{\epsilon_j}.$$

Now, let us define $\limsup_{\epsilon \to 0} A^{\epsilon}$ and $\liminf_{\epsilon \to 0} A^{\epsilon}$ as
$$
\limsup_{\epsilon \to 0} A^{\epsilon} = \displaystyle \bigcup_{\eps_j \to 0} \limsup_{j\to \infty} A^{\eps_j},
$$
that is, the smallest set that contains all possible limits along subsequences, and
$$
\liminf_{\epsilon \to 0} A^{\epsilon} = \bigcap_{\eps_j \to 0} \liminf_{j\to \infty} A^{\eps_j},
$$
that is, the largest set that is included in every possible sequential limit.

We have an upper bound for $\limsup_{\epsilon \to 0} A^{\epsilon}$.

\begin{lemma}
It holds that
$$\limsup_{\epsilon \to 0} A^{\epsilon} \subset A^u.$$
\end{lemma}

\begin{proof}
Let $K \subset \subset {\Omega} \setminus A^u$ and so $V=\Omega \setminus K$ is a neighborhood of $A^u$. There exists an $\eta$ such that $u-\Psi > \eta$ in $K$. By the uniform convergence there exists an $\eps_0$ depending on $K$ such that $u^{\eps} - \Psi >\eta/2$ for $\eps <\eps_0$.
Then we have for every $\epsilon <\epsilon_0$, $$A^{\eps} \subset  V.$$
Thus we have $$\limsup_{\epsilon_j \to 0} A^{\eps_j} \subset V$$
for any sequence $\epsilon_j \to 0$ and
for any neighborhood $V$ of $A^u$. Therefore $$\limsup_{\epsilon \to 0} A^{\eps} \subset A^u.$$
\end{proof}

To obtain a lower bound for $\liminf_{\epsilon \to 0} A^{\eps}$ we need to assume an extra condition on the obstacle.

\begin{lemma}
Assume that $\Psi$ satisfies
$
-\Delta_\infty \Psi(x_0) >0$ in the viscosity sense in $(A^u)^o$ then we
have $$\overline{(A^u)^o} \subset \liminf_{\epsilon \to 0} A^{\eps}.$$
\end{lemma}

\begin{proof}
Fix $x_0 \in (A^u)^o$ and choose any $\delta$ such that $B_\delta(x_0) \subset (A^u)^o$.
If $$u^{\eps_j} (x) > \Psi(x) \, \text{ for all } \, x\in B_\delta(x_0) \text{ and some sequence }\epsilon_j\to 0 $$ then
$$-\Delta_\infty^{\eps_j} u^{\eps_j}(x)=0 \, \text{ for all } \, x\in B_\delta(x_0) \text{ and every } \epsilon_j $$
therefore, by the argument in the proof of Theorem \ref{convergencia.teo.intro},
$$-\Delta_\infty u(x)=0 \, \text{ for } \, x\in B_{\delta/2}(x_0)$$
in the viscosity sense. As $B_{\delta/2}(x_0) \subset (A^u)^o$ we have that $u=\Psi$ there and hence we have
$$-\Delta_\infty \Psi(x_0)=0$$ a contradiction with our hypothesis. Therefore, for any sequence $\epsilon_j \to 0$
there exists $x_j \in A^{\eps_j}$ such that $x_j \to x_0$. Hence, $(A^u)^o \subset \liminf_{\eps_j\to 0} A^{\eps_j}$, and since $\liminf_{\eps_j\to 0} A^{\eps_j}$ is a closed set we get $\overline{(A^u)^o} \subset \liminf_{\eps_j\to 0} A^{\eps_j}$ for every sequence $\eps_j \to 0$.
Therefore $\overline{(A^u)^o} \subset \liminf_{\eps\to 0} A^{\eps}.$
\end{proof}

An immediate consequence of the previous two lemmas is the following result.

\begin{theorem}
Assume that $\Psi$ satisfies
$-\Delta_\infty \Psi(x_0) >0$ in the viscosity sense in $(A^u)^o$ and also assume that the contact set satisfies
$\overline{(A^u)^o} = A^u$,
 then we
have $$ \lim_{\epsilon \to 0} A^{\eps}=\liminf_{\eps\to 0} A^{\eps} =\limsup_{\eps\to 0} A^{\eps}= A^u.$$
\end{theorem}

\begin{remark} \label{rem.33} {\rm For a Lipschitz obstacle, it may happen that
$(A^u)^o = \emptyset$. In fact, take in $\Omega =B_1(0)$ and boundary function $F(x)=0$ and obstacle
$\Psi (x) = -3|x| +1$. The solution to the obstacle problem for the infinity Laplacian is given by the cone
$u(x) =  -|x| +1$ and the contact set $A^u$ is just a single point $A^u=\{0\}$.
}
\end{remark}

\end{section}
\medskip

\noindent{\bf Acknowledgments:} JDR was partially supported by
grants MTM2010-18128 and MTM2011-27998, Spain.

SJS would like to thank Luis Caffarelli for suggesting the problem to her.
SJS was partially supported by EMSW21-RTG - Program in Applied and Computational Analysis,
National Science Foundation DMS-0636586, NCE and the International Collaboratory for Emerging Technologies-CoLab,
Portuguese Science And Technology Foundation (FCT)
Project \#UTA06-894.

\end{document}